\documentclass[12pt,reqno]{amsart}
\usepackage{amscd,amssymb,amsmath,multicol,tikz,float}
\begin{document}

\restylefloat{table}
\newtheorem{thm}[equation]{Theorem}
\numberwithin{equation}{section}
\newtheorem{cor}[equation]{Corollary}
\newtheorem{expl}[equation]{Example}
\newtheorem{rmk}[equation]{Remark}
\newtheorem{conv}[equation]{Convention}
\newtheorem{claim}[equation]{Claim}
\newtheorem{lem}[equation]{Lemma}
\newtheorem{sublem}[equation]{Sublemma}
\newtheorem{conj}[equation]{Conjecture}
\newtheorem{defin}[equation]{Definition}
\newtheorem{diag}[equation]{Diagram}
\newtheorem{prop}[equation]{Proposition}
\newtheorem{notation}[equation]{Notation}
\newtheorem{tab}[equation]{Table}
\newtheorem{fig}[equation]{Figure}
\newcounter{bean}
\renewcommand{\theequation}{\thesection.\arabic{equation}}

\raggedbottom \voffset=-.7truein \hoffset=0truein \vsize=8truein
\hsize=6truein \textheight=8truein \textwidth=6truein
\baselineskip=18truept

\def\mapright#1{\ \smash{\mathop{\longrightarrow}\limits^{#1}}\ }
\def\mapleft#1{\smash{\mathop{\longleftarrow}\limits^{#1}}}
\def\mapup#1{\Big\uparrow\rlap{$\vcenter {\hbox {$#1$}}$}}
\def\mapdown#1{\Big\downarrow\rlap{$\vcenter {\hbox {$\ssize{#1}$}}$}}
\def\mapne#1{\nearrow\rlap{$\vcenter {\hbox {$#1$}}$}}
\def\mapse#1{\searrow\rlap{$\vcenter {\hbox {$\ssize{#1}$}}$}}
\def\mapr#1{\smash{\mathop{\rightarrow}\limits^{#1}}}
\def\ss{\smallskip}
\def\s{\sigma}
\def\l{\lambda}
\def\vp{v_1^{-1}\pi}
\def\at{{\widetilde\alpha}}

\def\sm{\wedge}
\def\la{\langle}
\def\ra{\rangle}
\def\ev{\text{ev}}
\def\od{\text{od}}
\def\on{\operatorname}
\def\ol#1{\overline{#1}{}}
\def\spin{\on{Spin}}
\def\cat{\on{cat}}
\def\Lbar{\overline{\Lambda}}
\def\qed{\quad\rule{8pt}{8pt}\bigskip}
\def\ssize{\scriptstyle}
\def\a{\alpha}
\def\bz{{\Bbb Z}}
\def\Rhat{\hat{R}}
\def\im{\on{im}}
\def\ct{\widetilde{C}}
\def\ext{\on{Ext}}
\def\sq{\on{Sq}}
\def\eps{\epsilon}
\def\ar#1{\stackrel {#1}{\rightarrow}}
\def\br{{\bold R}}
\def\bC{{\bold C}}
\def\bA{{\bold A}}
\def\bB{{\bold B}}
\def\bD{{\bold D}}
\def\bC{{\bold C}}
\def\bh{{\bold H}}
\def\bQ{{\bold Q}}
\def\bP{{\bold P}}
\def\bx{{\bold x}}
\def\bo{{\bold{bo}}}
\def\dh{\widehat{d}}
\def\si{\sigma}
\def\Vbar{{\overline V}}
\def\dbar{{\overline d}}
\def\wbar{{\overline w}}
\def\Sum{\sum}
\def\tfrac{\textstyle\frac}

\def\tb{\textstyle\binom}
\def\Si{\Sigma}
\def\w{\wedge}
\def\equ{\begin{equation}}
\def\b{\beta}
\def\G{\Gamma}
\def\L{\Lambda}
\def\g{\gamma}
\def\d{\delta}
\def\k{\kappa}
\def\cE{\mathcal{E}}
\def\psit{\widetilde{\Psi}}
\def\tht{\widetilde{\Theta}}
\def\psiu{{\underline{\Psi}}}
\def\thu{{\underline{\Theta}}}
\def\aee{A_{\text{ee}}}
\def\aeo{A_{\text{eo}}}
\def\aoo{A_{\text{oo}}}
\def\aoe{A_{\text{oe}}}
\def\vbar{{\overline v}}
\def\endeq{\end{equation}}
\def\sn{S^{2n+1}}
\def\zp{\bold Z_p}
\def\cR{{\mathcal R}}
\def\P{{\mathcal P}}
\def\cQ{{\mathcal Q}}
\def\cj{{\cal J}}
\def\zt{{\bold Z}_2}
\def\bs{{\bold s}}
\def\bof{{\bold f}}
\def\bq{{\bold Q}}
\def\be{{\bold e}}
\def\Hom{\on{Hom}}
\def\ker{\on{ker}}
\def\kot{\widetilde{KO}}
\def\coker{\on{coker}}
\def\da{\downarrow}
\def\colim{\operatornamewithlimits{colim}}
\def\zphat{\bz_2^\wedge}
\def\io{\iota}
\def\om{\omega}
\def\Prod{\prod}
\def\e{{\cal E}}
\def\zlt{\Z_{(2)}}
\def\exp{\on{exp}}
\def\abar{{\overline a}}
\def\xbar{{\overline x}}
\def\ybar{{\overline y}}
\def\zbar{{\overline z}}
\def\mbar{{\overline m}}
\def\nbar{{\overline n}}
\def\sbar{{\overline s}}
\def\kbar{{\overline k}}
\def\bbar{{\overline b}}
\def\et{{\widetilde E}}
\def\ni{\noindent}
\def\tsum{\textstyle \sum}
\def\coef{\on{coef}}
\def\den{\on{den}}
\def\lcm{\on{l.c.m.}}
\def\Ext{\operatorname{Ext}}
\def\iso{\approx}
\def\lra{\longrightarrow}
\def\vi{v_1^{-1}}
\def\ot{\otimes}
\def\psibar{{\overline\psi}}
\def\thbar{{\overline\theta}}
\def\mhat{{\hat m}}
\def\exc{\on{exc}}
\def\ms{\medskip}
\def\ehat{{\hat e}}
\def\etao{{\eta_{\text{od}}}}
\def\etae{{\eta_{\text{ev}}}}
\def\dirlim{\operatornamewithlimits{dirlim}}
\def\gt{\widetilde{L}}
\def\lt{\widetilde{\lambda}}
\def\st{\widetilde{s}}
\def\ft{\widetilde{f}}
\def\sgd{\on{sgd}}
\def\lfl{\lfloor}
\def\rfl{\rfloor}
\def\ord{\on{ord}}
\def\gd{{\on{gd}}}
\def\rk{{{\on{rk}}_2}}
\def\nbar{{\overline{n}}}
\def\MC{\on{MC}}
\def\lg{{\on{lg}}}
\def\cH{\mathcal{H}}
\def\cS{\mathcal{S}}
\def\cP{\mathcal{P}}
\def\N{{\Bbb N}}
\def\Z{{\Bbb Z}}
\def\Q{{\Bbb Q}}
\def\R{{\Bbb R}}
\def\C{{\Bbb C}}
\def\Lb{\overline\Lambda}
\def\mo{\on{mod}}
\def\xt{\times}
\def\notimm{\not\subseteq}
\def\Remark{\noindent{\it  Remark}}
\def\kut{\widetilde{KU}}
\def\Eb{\overline E}
\def\*#1{\mathbf{#1}}
\def\0{$\*0$}
\def\1{$\*1$}
\def\22{$(\*2,\*2)$}
\def\33{$(\*3,\*3)$}
\def\ss{\smallskip}
\def\ssum{\sum\limits}
\def\dsum{\displaystyle\sum}
\def\la{\langle}
\def\ra{\rangle}
\def\on{\operatorname}
\def\proj{\on{proj}}
\def\od{\text{od}}
\def\ev{\text{ev}}
\def\o{\on{o}}
\def\U{\on{U}}
\def\lg{\on{lg}}
\def\a{\alpha}
\def\bz{{\Bbb Z}}
\def\eps{\varepsilon}
\def\bc{{\bold C}}
\def\bN{{\bold N}}
\def\bB{{\bold B}}
\def\bW{{\bold W}}
\def\nut{\widetilde{\nu}}
\def\tfrac{\textstyle\frac}
\def\b{\beta}
\def\G{\Gamma}
\def\g{\gamma}
\def\zt{{\Bbb Z}_2}
\def\zth{{\bold Z}_2^\wedge}
\def\bs{{\bold s}}
\def\bx{{\bold x}}
\def\bof{{\bold f}}
\def\bq{{\bold Q}}
\def\be{{\bold e}}
\def\lline{\rule{.6in}{.6pt}}
\def\xb{{\overline x}}
\def\xbar{{\overline x}}
\def\ybar{{\overline y}}
\def\zbar{{\overline z}}
\def\ebar{{\overline \be}}
\def\nbar{{\overline n}}
\def\ubar{{\overline u}}
\def\bbar{{\overline b}}
\def\et{{\widetilde e}}
\def\M{\mathcal{M}}
\def\lf{\lfloor}
\def\rf{\rfloor}
\def\ni{\noindent}
\def\ms{\medskip}
\def\Dhat{{\widehat D}}
\def\what{{\widehat w}}
\def\Yhat{{\widehat Y}}
\def\abar{{\overline{a}}}
\def\minp{\min\nolimits'}
\def\sb{{$\ssize\bullet$}}
\def\mul{\on{mul}}
\def\N{{\Bbb N}}
\def\Z{{\Bbb Z}}
\def\S{\Sigma}
\def\Q{{\Bbb Q}}
\def\R{{\Bbb R}}
\def\C{{\Bbb C}}
\def\Xb{\overline{X}}
\def\eb{\overline{e}}
\def\notint{\cancel\cap}
\def\cS{\mathcal S}
\def\cR{\mathcal R}
\def\el{\ell}
\def\TC{\on{TC}}
\def\GC{\on{GC}}
\def\wgt{\on{wgt}}
\def\Ht{\widetilde{H}}
\def\wbar{\overline w}
\def\dstyle{\displaystyle}
\def\Sq{\on{sq}}
\def\Om{\Omega}
\def\ds{\dstyle}
\def\tz{tikzpicture}
\def\zcl{\on{zcl}}
\def\bd{\bold{d}}
\def\cM{\mathcal{M}}
\def\io{\iota}
\def\Vb#1{{\overline{V_{#1}}}}
\def\Ebar{\overline{E}}
\def\lb{\,\begin{picture}(-1,1)(1,-1)\circle*{4.5}\end{picture}\ }
\def\lbb{\,\begin{picture}(-1,1)(1,-1)\circle*{8}\end{picture}\ }
\def\zp{\Z_p}
\def\bL{\mathbf{L}}
\def\st{1.732}

\title
{Geodesic complexity of a cube}
\author{Donald M. Davis}
\address{Department of Mathematics, Lehigh University\\Bethlehem, PA 18015, USA}
\email{dmd1@lehigh.edu}

\date{August 8, 2023}

\maketitle
\begin{abstract} The topological (resp.~geodesic) complexity of a topological (resp.~metric) space is roughly the smallest number of continuous rules required to choose paths (resp.~shortest paths) between any points of the space. We prove that the geodesic complexity of a cube exceeds its topological complexity by exactly 2. The proof involves a careful analysis of cut loci of the cube.\end{abstract}

\section{Introduction}\label{introsec}
In \cite{Far}, Farber introduced the concept of the {\it topological complexity}, $\TC(X)$, of a topological space $X$, which is the minimal number $k$ such that there is a partition
$$X\times X=E_1\sqcup \cdots\sqcup E_k$$ 
with each $E_i$ being locally compact and admitting a continuous function $\phi_i:E_i\to P(X)$
such that $\phi_i(x_0,x_1)$ is a path from $x_0$ to 
$x_1$.
Here $P(X)$ is the space of paths in $X$ with the compact-open topology, and each $\phi_i$ is called a motion-planning rule. If $X$ is the space of configurations of one or more robots, this models the number of continuous rules required to program the robots to move between any two configurations. 

In \cite{david}, Recio-Mitter suggested that if $X$ is a metric space, then we require that the paths $\phi_i(x_0,x_1)$ be minimal geodesics (shortest paths) from $x_0$ to $x_1$, and defined the {\it geodesic complexity}, $\GC(X)$, to be the smallest number $k$ such that there is a partition 
$$X\times X= E_1\sqcup\cdots\sqcup E_k$$ 
with each $E_i$ being locally compact and admitting a continuous function $\phi_i:E_i\to P(X)$ such that $\phi_i(x_0,x_1)$ is a minimal geodesic from $x_0$ to $x_1$.\footnote{Recio-Mitter's  definition of $\GC(X)=k$ involved partitions into  sets $E_0,\ldots,E_k$, which, for technical reasons, has become the more common definition of concepts of this sort, but we prefer here to stick with Farber's more intuitive formulation.} Each function $\phi_i$ is called a {\it geodesic motion-planning rule} (GMPR).

One example discussed by Recio-Mitter in \cite{david} was when $X$ is (the surface of) a cube. It is well-known that here $\TC(X)=\TC(S^2)=3$, and he showed that $\GC(X)\ge4$. In this paper we prove that in this  case $\GC(X)=5$.

\begin{thm} \label{mainthm}If $X$ is a cube, then $\GC(X)=5$.
\end{thm}

For comparison, in \cite{tet} the author proved that for a regular tetrahedron $T$, $\GC(T)=4$ or 5, but was not able to establish the precise value. Here again $\TC(T)=\TC(S^2)=3$.

Our work relies heavily on the work of the author and Guo in \cite{DG}, where they analyzed the isomorphism classes as labeled graphs of cut loci on the cube. In Section \ref{background}, we review the relevant parts of that work.
In Section \ref{GMPR}, we prove that $\GC(X)\le5$ by constructing five explicit geodesic motion planning rules.
In Section \ref{lower}, we prove $\GC(X)\ge5$, using methods similar to those used in \cite{david} and \cite{tet}.

\section{Background on cut loci of a cube}\label{background}
In this section we present background material, mostly from \cite{DG}, regarding cut loci for a cube.

The {\it cut locus} of a point $P$ on a polyhedron is the closure of the set of points $Q$ such that there is more than one shortest path (minimal geodesic) from $P$ to $Q$. The cut locus is a labeled graph with corner points of the polyhedron labeling the leaves and perhaps  other vertices. Two labeled graphs are isomorphic if there is a graph bijection between them preserving labels. We let $\bL$ denote the isomorphism class of a cut locus.

Figure \ref{figA}, from \cite{DG}, shows the partition of a face of a cube into 193 connected subsets with constant $\bL$. Figure \ref{figB}, also from \cite{DG}, is a reparametrized version of the regions in the left quadrant of Figure \ref{figA}.

\bigskip

\begin{minipage}{6in}

\bigskip
\begin{center}

    \includegraphics[scale=0.22]{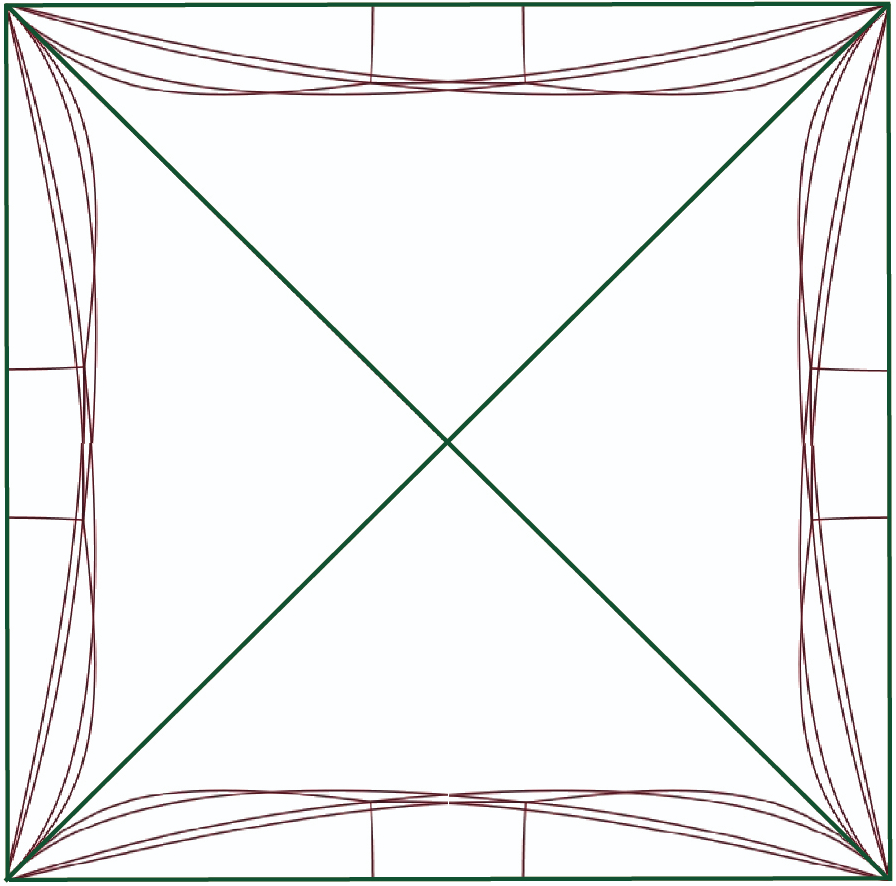}
   
   \begin{fig} \label{figA}
   
 {\bf Decomposition of a face into subsets on which $\textbf{L}$ is constant}
\\

\end{fig}
\end{center}
\end{minipage}
\\
\bigskip

\begin{minipage}{6in}

   \begin{center}
   
    \includegraphics[scale=0.09]{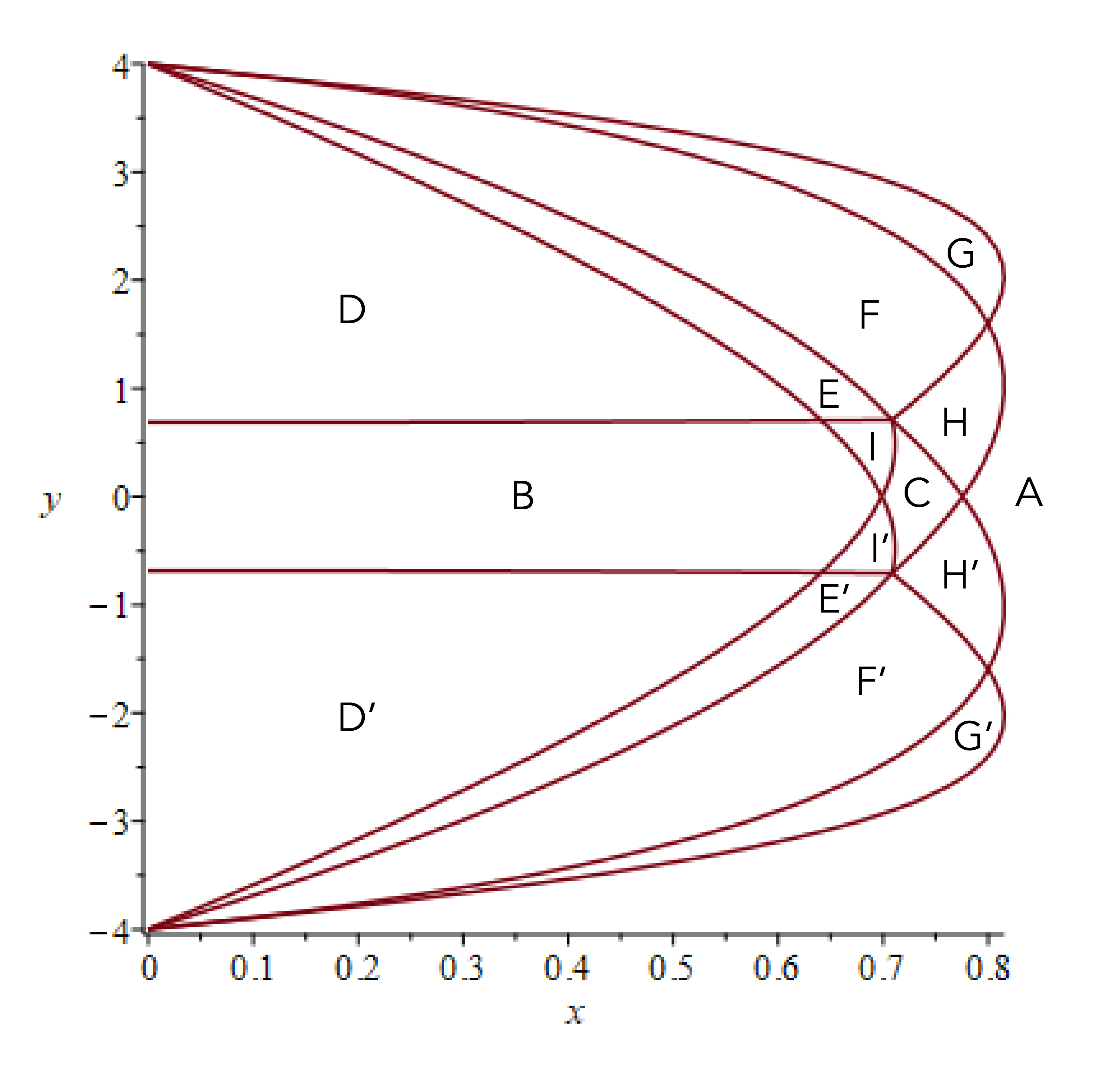}

    \begin{fig}\label{figB}

{\bf Regions in left quadrant of Figure \ref{figA}}
    
    \end{fig}
    
\end{center}
\end{minipage}

\bigskip
In \cite{DG} we listed, in stylized form, the $\bL$ for the various regions, but here, as we are interested in continuity of motion-planning rules, we are concerned about other aspects, such as the placement of edges of the cut locus with respect to one another.

The cut loci are found by the method of star unfolding and Voronoi diagrams, as developed in \cite{star97} and \cite{cut21}. We will use the same numbering of the corner points of the cube as was used in \cite{DG} and appears in Figure \ref{cube}, also taken from \cite{DG}, which, for future reference, includes an example of the cut locus of the midpoint of  edge 5-8.

\bigskip
\begin{minipage}{6in}

\bigskip
\begin{center}

    \includegraphics[scale=0.18]{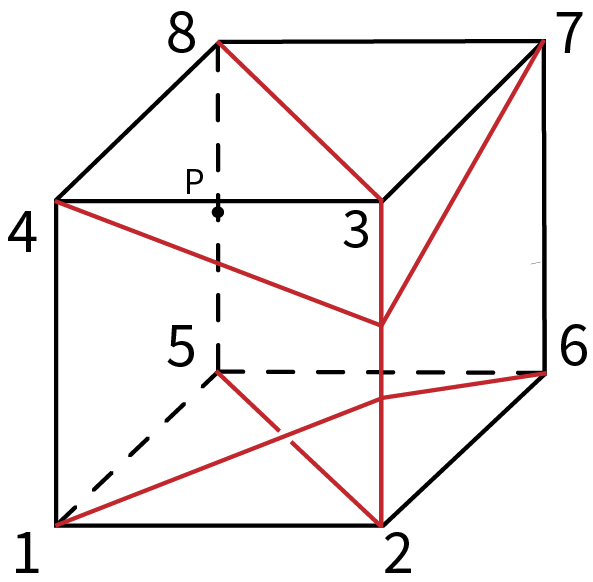}
   
   \begin{fig} \label{cube}
   
 {\bf A cube with labeled corner points, and the cut locus for the middle point of an edge highlighted}
\\

\end{fig}
\end{center}
\end{minipage}

\bigskip

In \cite{DG}, we explain how the diagram on the right side of Figure \ref{cutloc} is obtained, depicting in bold red the cut locus of the point $P$ in the left side of Figure \ref{cutloc}. The numbers at half of the vertices of the polygon correspond to the corner points in Figure \ref{cube}, and the labels $P_1,\ldots,P_8$ at the other vertices are different positions of the point $P$ in an unfolding of the cube. Every point of the cube occurs exactly once inside or on the 16-gon in Figure \ref{cutloc}, except that some occur on two boundary segments, and $P$ occurs eight times.

\begin{minipage}{6in}

\begin{center}

    \begin{tikzpicture}[scale=0.7]
    \draw[gray, thick] (1,1) -- (-1,1);
    \draw[gray, thick] (-1,1) -- (-1,-1);
    \draw[gray, thick] (1,-1) -- (-1,-1);
    \draw[gray, thick] (1,1) -- (1,-1);
    \draw (-9,-1) -- (-7,-1) -- (-7,1) -- (-9,1) -- (-9,-1);
     \draw[gray, thick] (-1,1) -- (-2.25,3.25);
     \draw[gray, thick] (-1,1) -- (-3.25,-0.25);
     \draw[gray, thick] (-3.25,-0.25)--(-1,-1);
     \draw[gray, thick] (-1,-1)--(-1.75,-3.25);
      \draw[gray, thick] (-2.25,3.25)--(-1,3);
      \draw[gray, thick] (-1,3)--(-0.75,4.25);
      \draw[gray, thick] (-0.75,4.25)--(1,1);
      \draw [gray, thick,dash dot] (-1,-3) -- (-1,-1);
      \draw [gray, thick,dash dot] (-1,3) -- (-1,1);
       \draw[gray, thick] (-1.75,-3.25)--(-1,-3);
      \draw[gray, thick] (-1,-3)--(-0.75,-3.75);
      \draw[gray, thick] (-0.75,-3.75)--(1,-1);
      \draw[gray, thick] (1,-1)--(3.75,-2.75);
       \draw[gray, thick] (3.75,-2.75)--(3,-1);
       \draw[gray, thick] (3,-1)--(4.75,-0.25);
       \draw[gray, thick] (3,1)--(4.75,-0.25);
       \draw [gray, thick,dash dot] (1,1) -- (3,1);
       \draw [gray, thick,dash dot] (1,-1) -- (3,-1);
       \draw[gray, thick] (3,1)--(4.25,2.75);
       \draw[gray, thick] (4.25,2.75)--(1,1);
       \coordinate[label={$1$}]  (1) at (-1.5,0.75);
       \coordinate[label={$2$}]  (2) at (1.1,1.1);
       \coordinate[label={$3$}]  (2) at (1,-1.75);
       \coordinate[label={$4$}]  (3) at (-1.5,-1.5);
       \coordinate[label={$5$}]  (5) at (-1.25,3.0);
       \coordinate[label={$6$}]  (6) at (3.5,0.75);
       \coordinate[label={$7$}]  (7) at (3.5,-1.5);
       \coordinate[label={$8$}]  (8) at (-1.25,-3.75);
       \coordinate[label={$P_1$}]  (p_1) at (-3.75,-0.75);
       \coordinate[label={$P_2$}]  (p_2) at (-2.75,3);
       \coordinate[label={$P_3$}]  (p_3) at (-0.75,4.15);
       \coordinate[label={$P_4$}]  (p_4) at (4.5,2.75);
        \coordinate[label={$P_5$}]  (p_5) at (5.25,-0.75);
         \coordinate[label={$P_6$}]  (p_6) at (4,-3.5);
          \coordinate[label={$P_7$}]  (p_7) at (-0.75,-4.75);
           \coordinate[label={$P_8$}]  (p_3) at (-2.25,-3.5);
            \draw[red, very thick](0.647,.529)--(-1,1);
         \draw[red, very thick](0.647,0.529)--(-1,3);
         \draw[red, very thick](0.333,-0.333)--(-1,-3);
         \draw[red, very thick](0.333,-0.333)--(-1,-1);
         \draw[red, very thick](0.895,0.649)--(1,1);
          \draw[red, very thick](0.895,0.649)--(3,1);
          \draw[red, very thick](0.805,-0.122)--(3,-1);
          \draw[red, very thick](0.805,-0.122)--(1,-1);
          \draw[red, very thick](0.805,-0.122)--(0.75,-0.0357);
          \draw[red, very thick](0.333,-0.333)--(0.75,-0.0357);
          \draw[red, very thick](0.647,0.529)--(0.75,0.4722);
           \draw[red, very thick](0.75,-0.0357)--(0.75,0.4722);
           \draw[red, very thick](0.895,0.649)--(0.75,0.4722);
           \node at (.6,.2) {$I$};
           \node at (-9.15,-1) {$5$};
           \node at (-6.85,-1) {$6$};
           \node at (-6.85,1) {$7$};
           \node at (-9.15,1) {$8$};
           \filldraw (-8.75,0.25) circle[radius=0.8pt];
           \node at (-8.5,.25) {$P$};
\end{tikzpicture}
\end{center}

\begin{fig} \label{cutloc}
{\bf Voronoi cells and cut locus of $P$}
\end{fig}
\end{minipage}
\bigskip

For example, the region in the right side of the 16-gon in Figure \ref{cutloc} bounded above and below by the segments coming in from the vertices labeled 6 and 7, on the right by $P_5$, and on the left by the short vertical segment $I$ is all the points that are closer to the $P_5$ version of point $P$ than to the others. This is called the Voronoi cell of $P_5$. The segment $I$ is equally close to versions $P_1$ and $P_5$. There are two equal minimal geodesics from $P$ to points on $I$; one crosses the segment connecting corner points 1 and 4, while the other crosses the segment connecting 6 and 7.

It is proved in \cite{DG} that the top and bottom halves of cut loci of the cube can be considered separately. Although all the regions in Figure \ref{figB} have distinct $\bL$, some have isomorphic top halves. For example, as can be seen in \cite[Figure 2.2]{DG}, regions $F$, $E$, $I$, $C$, and $H$ all have isomorphic top halves. We combine these here into a single region, which we will also call $F$. Similarly regions $D$, $B$, and $I'$ in Figure \ref{figB} have the same top half of $\bL$ and are combined into a single region, $D$. Also $D'$ and $E'$ combine to form $D'$, $F'$ and $G'$ combine to $F'$, and $A$, $G$, and $H'$ combine into $A$. This simplifies Figure \ref{figB} into our schematic Figure \ref{DF}, which only concerns top halves of $\bL$. We will discuss bottom halves later in this section.

\bigskip
\begin{minipage}{6in}

\begin{center}

    \begin{tikzpicture}[scale=.7]
\draw (0,0) -- (10,5) -- (0,10) -- (0,0);
\draw (0,4) -- (4,4);
\draw (0,0) to[out=30,in=270] (4,4);
\draw (0,0) to[out=30,in=325] (4,4);
\draw (4,4) to[out=40,in=-30] (0,10);
\draw (4,4) to[out=110,in=-30] (0,10);
\node at (7,5) {$A$};
\node at (1.5,6) {$D$};
\node at (1.5,3) {$D'$};
\node at (3.8,6) {$F$};
\node at (4.1,3) {$F'$};
\node at (4.2,4) {$*$};
\node at (-.3,5) {$\cE$};
\end{tikzpicture}
\end{center}

\begin{fig} \label{DF}
{\bf Regions with same top half of cut locus}
\end{fig}
\end{minipage}
\bigskip

There are also curves $DF$, $FA$, $DD'$, $D'F'$, and $F'A$ bounding these combined regions. There is also $*$, the intersection point, and the left edge $\cE$. In Figure \ref{big1}, we depict the top half of the cut loci for these regions, with arrows indicating convergence of points in a region to points in its boundary, in each of which an edge of the graph is collapsed.

\bigskip
\begin{minipage}{6in}

\begin{center}

    \begin{tikzpicture}[scale=.6]
\draw (1,11) -- (1,13) -- (2,14);
\draw (7,11) -- (7,14);
\draw (14,-3) -- (14,-2.3) -- (12,0);
\draw (13,11) -- (13,11.7) -- (15,14);
\draw (20,11) -- (20,14);
\node at (4,12.5) {$\to$};
\node at (4,5.5) {$\leftarrow$};
\node at (10.3,12.5) {$\leftarrow$};
\node at (10.3,5.5) {$\to$};
\node at (17,12.5) {$\to$};
\node at (16,5.5) {$\leftarrow$};
\node at (4,-1.5) {$\to$};
\node at (10.3,-1.5) {$\leftarrow$};
\node at (17,-1.5) {$\to$};
\node at (1,10.3) {$D$};
\node at (7,10.3) {$DF$};
\node at (13,10.3) {$F$};
\node at (20,10.3) {$FA$};
\draw (1,11.8) -- (0,11.8);
\draw (1,12.5) -- (2,12.5);
\draw (1,13) -- (0,14);
\node at (-.15,11.8) {$1$};
\node at (2.15,12.5) {$6$};
\node at (2,14.15) {$2$};
\node at (0,14.15) {$5$};
\draw (7,12) -- (6,12);
\draw (6,13.2) -- (7,13) -- (8,13.2);
\node at (5.8,12) {$1$};
\node at (5.8,13.2) {$5$};
\node at (8.2,13.2) {$6$};
\node at (7,14.2) {$2$};
\draw (14,-2.3) -- (15,-2.3);
\draw (13.3,-1.5) -- (14.3,-1);
\draw (12.6,-.7) -- (13.2,0);
\node at (15.2,-2.3) {$6$};
\node at (14.5,-1) {$2$};
\node at (13.3,.1) {$5$};
\node at (12,.2) {$1$};
\draw (19,12) -- (20,12) -- (19.3,12.8);
\draw (20,13) -- (21,13);
\node at (18.8,12) {$1$};
\node at (19.1,12.8) {$5$};
\node at (21.2,13) {$6$};
\node at (20,14.2) {$2$};
\draw (1,4) -- (1,6) -- (0,7);
\draw (0,5) -- (2,5);
\node at (1,3.3) {$\cE$};
\node at (-.2,5) {$1$};
\node at (2.2,5) {$6$};
\node at (0,7.2) {$5$};
\node at (1,8.8) {$\downarrow$};
 \filldraw (1,6) circle[radius=2pt];
 \node at (3.5,8.8) {$\searrow$};
 \node at (1.3,6) {$2$};
 \node at (7,3.3) {$DD'$};
 \node at (13,3.3) {$*$};
 \node at (20,3.3) {$A$};
 \draw (7,4) -- (7,6) -- (8,7);
 \draw (7,6) -- (6,7);
 \draw (6,5) -- (8,5);
 \node at (5.75,5) {$1$};
 \node at (8.25,5) {$6$};
 \node at (6,7.25) {$5$};
 \node at (8,7.25) {$2$};
 \draw (13,4) -- (13,5.5) -- (14,5.5);
 \draw (12,5.5) -- (13,5.5) -- (12,7);
 \draw (13,5.5) -- (14,7);
 \node at (11.75,5.5) {$1$};
 \node at (14.25,5.5) {$6$};
 \node at (12,7.25) {$5$};
 \node at (14,7.25) {$2$};
 \draw (20,4) -- (20,5) -- (21,6) -- (21,7);
 \draw (21,6) -- (22,6);
 \draw (20,5) -- (19,6) -- (19,7);
 \draw (19,6) -- (18,6);
 \node at (9.5,8.8) {$\searrow$};
 \node at (13,8.8) {$\downarrow$};
 \node at (20,8.8) {$\uparrow$};
 \node at (22.25,6) {$6$};
 \node at (17.75,6) {$1$};
 \node at (21,7.25) {$2$};
 \node at (19,7.25) {$5$};
 \node at (1,-3.7) {$D'$};
 \node at (7,-3.7) {$D'F'$};
 \node at (14,-3.7) {$F'$};
 \node at (20,-3.7) {$F'A$};
 \node at (9.5,1.8) {$\nearrow$};
 \node at (1,1.8) {$\uparrow$};
 \node at (13,1.8) {$\uparrow$};
 \node at (20,1.8) {$\downarrow$};
 \node at (4,1.8) {$\nearrow$};
 \draw (1,-3) -- (1,-1) -- (2,0);
 \draw (1,-1) -- (0,0);
 \draw (1,-2.3) -- (2,-2.3);
 \draw (1,-1.6) -- (0,-1.6);
 \node at (2.25,-2.3) {$6$};
 \node at (-.25,-1.6) {$1$};
 \node at (0,.25) {$5$};
 \node at (2,.25) {$2$};
 \draw (7,-3) -- (7,0);
 \draw (7,-2) -- (8,-2);
 \draw (6,-.8) -- (7,-1) -- (8,-.8);
 \node at (8.25,-2) {$6$};
 \node at (5.75,-.8) {$1$};
 \node at (8.25,-.8) {$2$};
 \node at (7,.25) {$5$};
 \draw (20,-3) -- (20,0);
 \draw (21,-2) -- (20,-2) -- (21,-1);
 \draw (20,-1) -- (19,-1);
 \node at (21.25,-2) {$6$};
 \node at (21.25,-1) {$2$};
 \node at (18.75,-1) {$1$};
 \node at (20,.25) {$5$};
  \node at (20,1.8) {$\downarrow$};
 \node at (4,1.8) {$\nearrow$};
 \node at (16.5,1.8) {$\nwarrow$};
 \node at (16.5,8.8) {$\swarrow$};
 \draw [->] (5.3,9.4) -- (8.5,8.2);
 \draw [->] (5.3,1.2) -- (8.5,2.4);\draw (13,11.7) -- (12,11.7);
\draw (13.7,12.5) -- (12.7,13);
\draw (14.4,13.3) -- (13.8,14);
\node at (11.8,11.7) {$1$};
\node at (12.5,13) {$5$};
\node at (13.7,14.1) {$2$};
\node at (15,14.2) {$6$};
\end{tikzpicture}
\end{center}

\begin{fig} \label{big1}
{\bf Top halves of cut loci}
\end{fig}
\end{minipage}
\bigskip

The bottom half of the cut locus of a point in a region $R$ in Figure \ref{figB} is obtained from the top half of the cut locus of the vertical reflection of the point, which is in reflected region $R'$, by inverting it and applying the permutation $(1\ 4)(2\ 3)(5\ 8)(6\ 7)$ to the labels. The collecting of several regions of Figure \ref{figB} into a single region  with the same bottom half of $\bL$ is essentially a vertical flip of what was done in forming Figure \ref{DF} for top halves. For example, the vertical reflection of the region $D'$ of Figure \ref{DF} contains regions $D$ and $E$ of Figure \ref{figB}, and its cut locus bottom is as in Figure \ref{bot}.

\bigskip
\begin{minipage}{6in}

\begin{center}

    \begin{tikzpicture}[scale=.7]
\draw (0,1) -- (0,4);
\draw (-1,0) -- (0,1) -- (1,0);
\draw (0,3) -- (1,3);
\draw (0,2) -- (-1,2);
\node at (-1.25,2) {$4$};
\node at (1.25,3) {$7$};
\node at (-1,-.25) {$8$};
\node at (1,-.25) {$3$};
\end{tikzpicture}
\end{center}

\begin{fig} \label{bot}
{\bf Cut locus bottom of flip of  region $D'$ of Figure \ref{DF}}
\end{fig}
\end{minipage}
\bigskip

Each region in the top quadrant of Figure \ref{figA} is obtained from the corresponding region in the left quadrant by a clockwise rotation of $\pi/2$ around the center of the square. The cut locus of the new region is obtained from that of the old one by applying the permutation $(1\ 4\ 3\ 2)(5\ 8\ 7\ 6)$ to the labels and then rotating the resulting figure $\pi/2$ counter-clockwise. In Figure \ref{AA} we show the cut locus of points in region $A$, in the rotated region $A_R$, and in the half-diagonal separating them.

\bigskip
\begin{minipage}{6in}

\begin{center}

    \begin{tikzpicture}[scale=.5]
\draw (0,0) -- (5,0) -- (5,5) -- (0,5) -- (0,0) -- (5,5);
\draw (5,0) -- (0,5);
\draw (0,0) to[out=70,in=-70] (0,5);
\draw (0,5) to[out=-20,in=200] (5,5);
\node at (1,2.5) {$A$};
\node at (2.5,3.7) {$A_R$};
\node at (-.3,0) {$6$};
\node at (5.3,0) {$5$};
\node at (5.3,5) {$7$};
\node at (-.3,5) {$8$};
\draw (9,0) -- (9,1) -- (10,2) -- (11,1) -- (11,0);
\draw (8,1) -- (9,1);
\draw (11,1) -- (12,1);
\draw (10,2) -- (10,3) -- (11,4) -- (12,4);
\draw (11,4) -- (11,5);
\draw (10,3) -- (9,4) -- (9,5);
\draw (9,4) -- (8,4);
\node at (9,-.3) {$8$};
\node at (11,-.3) {$3$};
\node at (10,-1) {$A$};
\node at (7.7,1) {$4$};
\node at (12.3,1) {$7$};
\node at (12.3,4) {$6$};
\node at (7.7,4) {$1$};
\node at (9,5.3) {$5$};
\node at (11,5.3) {$2$};
\draw (15,0) -- (19,5);
\draw (19,0) -- (15,5);
\draw (18,1.25) -- (17,0);
\draw (16,3.75) -- (15,3);
 \filldraw (16,1.25) circle[radius=2pt];
  \filldraw (18,3.75) circle[radius=2pt];
\node at (15,-.3) {$8$};
\node at (17,-.3) {$3$};
\node at (19,-.3) {$7$};
\node at (17,-1){diag};
\node at (15.6,1.25) {$4$};
\node at (17.6,3.75) {$2$};
\node at (19,5.3) {$6$};
\node at (15,5.3) {$5$};
\node at (14.7,3) {$1$};
\draw (22,1) -- (23,1) -- (24,2) --  (25,2) -- (26,1) -- (27,1);
\draw (23,0) -- (23,1);
\draw (26,0) -- (26,1);
\draw (24,2) -- (23,3) -- (22,3);
\draw (23,3) -- (23,4);
\draw (25,2) -- (26,3) -- (27,3);
\draw (26,3) -- (26,4);
\node at (23,-.3) {$4$};
\node at (26,-.3) {$3$};
\node at (24.5,-1) {$A_R$};
\node at (21.7,1) {$8$};
\node at (27.3,1) {$7$};
\node at (27.3,3) {$2$};
\node at (21.7,3) {$1$};
\node at (23,4.3) {$5$};
\node at (26,4.3) {$6$};
\end{tikzpicture}
\end{center}

\begin{fig} \label{AA}
{\bf Cut locus of rotation of region}
\end{fig}
\end{minipage}
\bigskip

In \cite{DG}, we were only concerned about isomorphism type as a graph, but here we care about the relative positions of the labeled arms.

\section{Geodesic motion planning rules}\label{GMPR}
In this section, we construct five geodesic motion-planning rules for the cube. The remainder of this section is devoted to the proof of the following result.
\begin{thm}\label{thm1} If $X$ is the cube, then $X\times X$ can be partitioned into five locally-compact subsets $E_i$ with a GMPR $\phi_i$ on each.\end{thm}

We define $E_1$ to be the set of pairs $(P,Q)$ such that there is a unique minimal geodesic from $P$ to $Q$, and let $\phi_1(P,Q)$ be that path. It is well-known (\cite[Chapter 1, 3.12 Lemma]{BH}) that such a function is continuous. Note that a corner point $V$ at a leaf of the cut locus graph of a point $P$ is not in the cut locus, so these $(P,V)$ are in $E_1$.

We define the {\it multiplicity} of $(P,Q)$ (or of just $Q$ if $P$ is implicit) to be the number of distinct
minimal geodesics from $P$ to $Q$. If $Q$ is on an edge (resp.~is a vertex) of the cut locus graph of $P$, then the multiplicity of $(P,Q)$ equals 2 (resp.~the degree of the vertex).

We define $E_2$ to be the set of all $(P,Q)$ of multiplicity 2. The points $Q$ will, for the most part, be interiors of edges of the cut locus graph. It also includes any degree-2 vertex, such as vertex 2 in the cut locus of $\cE$ in Figure \ref{big1}. The function $\phi_2$ is defined using an orientation of the cube; i.e., a continuous choice of direction of rotation around each point. The cut locus of $P$ varies continuously with $P$, unless $P$ is a corner point. We will deal with the case with $P$ a corner point later. The cut loci of points in a quadrant is a tree consisting of two parts connected by a segment parallel to the edge of the quadrant. See, for example, the cut loci of points in regions $A$ and $A_R$ pictured in Figure \ref{AA}. For a 3-dimensional example, see Figure \ref{cube}.  For points on the diagonals separating quadrants, the connecting ``segment'' consists of a single point. Think of rotating the cut locus around the center of that segment in the direction given by the orientation. We define $\phi_2(P,Q)$ to be the geodesic from $P$ to $Q$ which approaches $Q$ in the direction of the rotation. We will deal with the connecting segments shortly.

In Figures \ref{AA2} and \ref{big2}, we add to Figures \ref{AA} and \ref{big1} red dots on the edges of several cut loci indicating the side from which $Q$ should be approached if the orientation is clockwise.

\bigskip
\begin{minipage}{6in}

\begin{center}

    \begin{tikzpicture}[scale=.5]
\draw (0,0) -- (5,0) -- (5,5) -- (0,5) -- (0,0) -- (5,5);
\draw (5,0) -- (0,5);
\draw (0,0) to[out=70,in=-70] (0,5);
\draw (0,5) to[out=-20,in=200] (5,5);
\node at (1,2.5) {$A$};
\node at (2.5,3.7) {$A_R$};
\node at (-.3,0) {$6$};
\node at (5.3,0) {$5$};
\node at (5.3,5) {$7$};
\node at (-.3,5) {$8$};
\draw (9,0) -- (9,1) -- (10,2) -- (11,1) -- (11,0);
\draw (8,1) -- (9,1);
\draw (11,1) -- (12,1);
\draw (10,2) -- (10,3) -- (11,4) -- (12,4);
\draw (11,4) -- (11,5);
\draw (10,3) -- (9,4) -- (9,5);
\draw (9,4) -- (8,4);
\node at (9,-.3) {$8$};
\node at (11,-.3) {$3$};
\node at (10,-1) {$A$};
\node at (7.7,1) {$4$};
\node at (12.3,1) {$7$};
\node at (12.3,4) {$6$};
\node at (7.7,4) {$1$};
\node at (9,5.3) {$5$};
\node at (11,5.3) {$2$};
\draw (15,0) -- (19,5);
\draw (19,0) -- (15,5);
\draw (18,1.25) -- (17,0);
\draw (16,3.75) -- (15,3);
 \filldraw (16,1.25) circle[radius=2pt];
  \filldraw (18,3.75) circle[radius=2pt];
\node at (15,-.3) {$8$};
\node at (17,-.3) {$3$};
\node at (19,-.3) {$7$};
\node at (17,-1){diag};
\node at (15.6,1.25) {$4$};
\node at (17.6,3.75) {$2$};
\node at (19,5.3) {$6$};
\node at (15,5.3) {$5$};
\node at (14.7,3) {$1$};
\draw (22,1) -- (23,1) -- (24,2) --  (25,2) -- (26,1) -- (27,1);
\draw (23,0) -- (23,1);
\draw (26,0) -- (26,1);
\draw (24,2) -- (23,3) -- (22,3);
\draw (23,3) -- (23,4);
\draw (25,2) -- (26,3) -- (27,3);
\draw (26,3) -- (26,4);
\node at (23,-.3) {$4$};
\node at (26,-.3) {$3$};
\node at (24.5,-1) {$A_R$};
\node at (21.7,1) {$8$};
\node at (27.3,1) {$7$};
\node at (27.3,3) {$2$};
\node at (21.7,3) {$1$};
\node at (23,4.3) {$5$};
\node at (26,4.3) {$6$};
\filldraw [color=red] ( 8.5,3.9) circle[radius=2pt];
\filldraw [color=red] (8.9,4.5 ) circle[radius=2pt];
\filldraw [color=red] ( 10.9,4.5) circle[radius=2pt];
\filldraw [color=red] ( 11.5,4.1) circle[radius=2pt];
\filldraw [color=red] ( 10.5,3.6) circle[radius=2pt];
\filldraw [color=red] (8.5,.9 ) circle[radius=2pt];
\filldraw [color=red] ( 9.1,.5) circle[radius=2pt];
\filldraw [color=red] (23.1,.5  ) circle[radius=2pt];
\filldraw [color=red] ( 26.1,.5 ) circle[radius=2pt];
\filldraw [color=red] ( 22.5,.9 ) circle[radius=2pt];
\filldraw [color=red] ( 26.5,1.1 ) circle[radius=2pt];
\filldraw [color=red] ( 23.5,1.4 ) circle[radius=2pt];
\filldraw [color=red] ( 25.5,1.6 ) circle[radius=2pt];
\filldraw [color=red] ( 23.5,2.4 ) circle[radius=2pt];
\filldraw [color=red] ( 25.5,2.6 ) circle[radius=2pt];
\filldraw [color=red] ( 22.5,2.9 ) circle[radius=2pt];
\filldraw [color=red] ( 26.5,3.1 ) circle[radius=2pt];
\filldraw [color=red] (  22.9,3.5) circle[radius=2pt];
\filldraw [color=red] ( 25.9,3.5 ) circle[radius=2pt];
\filldraw [color=red] ( 9.5,3.4 ) circle[radius=2pt];
\filldraw [color=red] ( 10.5,1.6 ) circle[radius=2pt];
\filldraw [color=red] (11.5,1.1  ) circle[radius=2pt];
\filldraw [color=red] ( 11.1,.5 ) circle[radius=2pt];
\filldraw [color=red] ( 9.5,1.4 ) circle[radius=2pt];
\filldraw [color=red] ( 15.4,4.38 ) circle[radius=2pt];
\filldraw [color=red] ( 16.4,3.13 ) circle[radius=2pt];
\filldraw [color=red] ( 17.4,3.13 ) circle[radius=2pt];
\filldraw [color=red] ( 18.4,4.38 ) circle[radius=2pt];
\filldraw [color=red] ( 15.5,3.28 ) circle[radius=2pt];
\filldraw [color=red] ( 15.62,.62 ) circle[radius=2pt];
\filldraw [color=red] ( 16.62,1.87 ) circle[radius=2pt];
\filldraw [color=red] ( 17.62,1.87 ) circle[radius=2pt];
\filldraw [color=red] ( 17.62,.62 ) circle[radius=2pt];
\filldraw [color=red] ( 18.62,.62 ) circle[radius=2pt];
\end{tikzpicture}
\end{center}

\begin{fig} \label{AA2}
{\bf Direction for $\phi_2$ for some cut loci}
\end{fig}
\end{minipage}
\bigskip

\bigskip
\begin{minipage}{6in}

\begin{center}

    \begin{tikzpicture}[scale=.6]
\node at (4,5.5) {$\leftarrow$};
\node at (10.3,5.5) {$\to$};
\node at (16,5.5) {$\leftarrow$};

\draw (1,4) -- (1,7);
\draw (0,5) -- (2,5);
\node at (1,3.3) {$\cE$};
\node at (-.2,5) {$1$};
\node at (2.2,5) {$6$};
\node at (1,7.2) {$5$};
 \filldraw (1,6) circle[radius=2pt];
 \node at (.74,6) {$2$};
 \node at (7,3.3) {$DD'$};
 \node at (13,3.3) {$*$};
 \node at (20,3.3) {$A$};
 \draw (7,4) -- (7,6) -- (8,7);
 \draw (7,6) -- (6,7);
 \draw (6,5) -- (8,5);
 \node at (5.75,5) {$1$};
 \node at (8.25,5) {$6$};
 \node at (6,7.25) {$5$};
 \node at (8,7.25) {$2$};
 \draw (13,4) -- (13,5.5) -- (14,5.5);
 \draw (12,5.5) -- (13,5.5) -- (12,7);
 \draw (13,5.5) -- (14,7);
 \node at (11.75,5.5) {$1$};
 \node at (14.25,5.5) {$6$};
 \node at (12,7.25) {$5$};
 \node at (14,7.25) {$2$};
 \draw (20,4) -- (20,5) -- (21,6) -- (21,7);
 \draw (21,6) -- (22,6);
 \draw (20,5) -- (19,6) -- (19,7);
 \draw (19,6) -- (18,6);
 \node at (22.25,6) {$6$};
 \node at (17.75,6) {$1$};
 \node at (21,7.25) {$2$};
 \node at (19,7.25) {$5$};
 \filldraw [color=red] ( .5,4.88 ) circle[radius=2pt];
 \filldraw [color=red] ( 1.5,5.12 ) circle[radius=2pt];
 \filldraw [color=red] ( .88,5.5 ) circle[radius=2pt];
 \filldraw [color=red] ( .88,6.5 ) circle[radius=2pt];
 \filldraw [color=red] (6.5,4.88  ) circle[radius=2pt];
 \filldraw [color=red] ( 7.5,5.12 ) circle[radius=2pt];
 \filldraw [color=red] ( 6.88,5.5 ) circle[radius=2pt];
 \filldraw [color=red] ( 6.5,6.38 ) circle[radius=2pt];
 \filldraw [color=red] ( 7.5,6.62 ) circle[radius=2pt];
 \filldraw [color=red] (12.5,5.38  ) circle[radius=2pt];
 \filldraw [color=red] ( 13.5,5.62 ) circle[radius=2pt];
 \filldraw [color=red] ( 12.5,6.13 ) circle[radius=2pt];
 \filldraw [color=red] ( 13.5,6.37 ) circle[radius=2pt];
 \filldraw [color=red] ( 19.5,5.38 ) circle[radius=2pt];
 \filldraw [color=red] ( 20.5,5.62 ) circle[radius=2pt];
 \filldraw [color=red] ( 18.5,5.88 ) circle[radius=2pt];
 \filldraw [color=red] ( 21.5,6.12 ) circle[radius=2pt];
 \filldraw [color=red] ( 20.88,6.5 ) circle[radius=2pt];
 \filldraw [color=red] ( 18.88,6.5 ) circle[radius=2pt];
\end{tikzpicture}
\end{center}

\begin{fig} \label{big2}
{\bf Direction for $\phi_2$ in some top halves}
\end{fig}
\end{minipage}

\bigskip

Regarding the connecting segments, note that  each edge of the cube bounds two quadrants, and all cut loci in those two quadrants have parallel connecting segments. Arbitrarily make a uniform choice of a  side of these segments. Let $\phi_2(P,Q)$ for $Q$ in those connecting segments be the minimal geodesic from $P$ to $Q$ which approaches $Q$ from the selected side. Because the quadrants are bounded by diagonals in which the connecting points of cut loci halves are vertices of degree 4 and so are not part of $E_2$, compatibility of the GMPRs for connecting segments in distinct quadrant-pairs is not an issue.

The cut locus of a corner point consists of the three edges and three diagonals emanating from the opposite corner point. Although it is not the case that the cut loci vary continuously with $P$ as $P$ approaches a corner point, we show that our defining $\phi_2$ using rotation around a central point is still continuous at the corner point. In Figure \ref{four}, we depict the cut loci of corner point $V_8$ and of points $P$ close to $V_8$ along the 5-8 edge,  along the curve $DE$ in Figure \ref{figB}, and along the diagonal, adorned with red dots indicating the direction from which the side should be approached using $\phi_2$.

\bigskip
\begin{minipage}{6in}

\begin{center}

    \begin{tikzpicture}[scale=.6]
\draw (0,-2) -- (2,0) -- (4,-2);
\draw (2,-2) -- (2,0) -- (4,0);
\draw (0,0) -- (2,0) -- (0,2);
\draw (7,-4) -- (9,-2) -- (8.9,-.3) -- (8.9,-.15) -- (9,0) -- (7,2);
\draw (7,0) -- (8.9,-.15) -- (11,0);
\draw (7,-2) -- (8.9,-.3) -- (11,-2);
\draw (14,-4) -- (15.83,-.85) -- (15.83,-.34) -- (14,0);
\draw (16,-2) -- (15.83,-.85) -- (15.83,-.68) -- (14,-2);
\draw (18,-2) -- (15.83,-.51) -- (15.83,-.17) -- (16,0);
\draw (18,0) -- (15.83,-.17) -- (14,2);
\draw (21,-4) -- (21,-2) -- (22.85,-.15) -- (23,0) -- (25,0);
\draw (23,-2) -- (22.85,-.3) -- (25,-2);
\draw (22.85,-.3) -- (22.85,-.15) -- (22.7,0) -- (21,2);
\draw (22.7,0) -- (21,0);
\node at (2,-4) {$V_8$};
\node at (9,-4) {edge};
\node at (16,-4) {$DE$};
\node at (23,-4) {diag};
\node at (0,-2.3) {$4$};
\node at (2,-2.3) {$3$};
\node at (4,-2.3) {$7$};
\node at (4.3,0) {$6$};
\node at (-.3,0) {$1$};
\node at (-.3,2) {$5$};
\node at (2.2,.3) {$2$};
\node at (7,-4.3) {$8$};
\node at (7,-2.3) {$4$};
\node at (11,-2.3) {$7$};
\node at (9.2,-2.2) {$3$};
\node at (11.3,0) {$6$};
\node at (6.7,0) {$1$};
\node at (9.2,.3) {$2$};
\node at (6.7,2) {$5$};
\node at (14,-4.3) {$8$};
\node at (14,-2.3) {$4$};
\node at (16,-2.3) {$3$};
\node at (18,-2.3) {$7$};
\node at (13.7,2) {$5$};
\node at (13.7,0) {$1$};
\node at (18.3,0) {$6$};
\node at (16.2,.3) {$2$};
\node at (21,-4.3) {$8$};
\node at (23,-2.3) {$3$};
\node at (20.7,-2) {$4$};
\node at (25,-2.3) {$7$};
\node at (25.3,0) {$6$};
\node at (20.7,0) {$1$};
\node at (20.7,2) {$5$};
\node at (23.2,.3) {$2$};
 \filldraw [color=red] (3,.12 ) circle[radius=2pt];
 \filldraw [color=red] ( .88,1) circle[radius=2pt];
 \filldraw [color=red] ( 1,-.12) circle[radius=2pt];
 \filldraw [color=red] ( 1,-1.12) circle[radius=2pt];
 \filldraw [color=red] ( 2.12,-1) circle[radius=2pt];
 \filldraw [color=red] (3.12,-1 ) circle[radius=2pt];
 \filldraw [color=red] ( 10,.1) circle[radius=2pt];
 \filldraw [color=red] ( 8,.88) circle[radius=2pt];
 \filldraw [color=red] ( 8,-.15) circle[radius=2pt];
 \filldraw [color=red] (8,-1.18 ) circle[radius=2pt];
 \filldraw [color=red] ( 8,-3.12) circle[radius=2pt];
 \filldraw [color=red] ( 9.09,-1.05) circle[radius=2pt];
 \filldraw [color=red] ( 10.08,-1) circle[radius=2pt];
 \filldraw [color=red] (8.8,0 ) circle[radius=2pt];
 \filldraw [color=red] ( 17,.05) circle[radius=2pt];
 \filldraw [color=red] ( 15,.65) circle[radius=2pt];
 \filldraw [color=red] ( 15,-.28) circle[radius=2pt];
 \filldraw [color=red] ( 15,-1.34) circle[radius=2pt];
 \filldraw [color=red] ( 15.15,-2.3) circle[radius=2pt];
 \filldraw [color=red] (16,-1.4 ) circle[radius=2pt];
 \filldraw [color=red] ( 17,-1.15) circle[radius=2pt];
 \filldraw [color=red] ( 15.83,.03) circle[radius=2pt];
 \filldraw [color=red] (15.95,-.78 ) circle[radius=2pt];
 \filldraw [color=red] (24,.13 ) circle[radius=2pt];
 \filldraw [color=red] ( 22,.68) circle[radius=2pt];
 \filldraw [color=red] ( 22,-.1) circle[radius=2pt];
 \filldraw [color=red] (22,-1.12 ) circle[radius=2pt];
 \filldraw [color=red] ( 21.12,-3) circle[radius=2pt];
 \filldraw [color=red] ( 23.09,-1.2) circle[radius=2pt];
 \filldraw [color=red] (24,-1.1 ) circle[radius=2pt];
 \filldraw [color=red] ( 22.92,0) circle[radius=2pt];

\end{tikzpicture}
\end{center}

\begin{fig} \label{four}
{\bf Cut locus of a corner point and of points near it}
\end{fig}
\end{minipage}
\bigskip

For $P$ on the edge, or $DE$, or the diagonal approaching $V_8$, the points $Q$ in the cut locus of $P$ on the segment emanating from vertex number 8 approach a point $Q_0$ which is not in the cut locus of $V_8$.
Then $(V_8,Q_0)$ is in $E_1$, and so we don't have to worry about the limit of $\phi_2(P,Q)$.

The set $E_3$ consists of the 56 points $(P,Q)$ such that $Q$ is a vertex of the cut locus of $P$ of degree 5 or 6. Since this is a discrete set, the function $\phi_3$ can be defined arbitrarily. Eight of these points have $P$ a corner point of the cube and $Q$ the opposite corner point. The cut locus of a corner point was depicted in the left side of Figure \ref{four}.

Another point in $E_3$ has $P$ equal to the point $*$, which was introduced in Figure \ref{DF}. The top half of its cut locus was shown in Figure \ref{big1}; we show its entire cut locus in Figure \ref{here}.

\bigskip
\begin{minipage}{6in}

\begin{center}

    \begin{tikzpicture}[scale=.7]
\draw (2,1) -- (1,1) -- (0,2) -- (0,4) -- (2,4);
\draw (0,2) -- (-1,1);
\draw (0,3) -- (-2,3);
\draw (-2,4) -- (0,4) -- (1,5);
\draw (0,4) -- (-1,5);
\draw (1,0) -- (1,1);
\node at (1,-.3) {$3$};
\node at (2.3,1) {$7$};
\node at (-1.2,.8) {$8$};
\node at (-2.3,3) {$4$};
\node at (-2.3,4) {$1$};
\node at (2.3,4) {$6$};
\node at (1.2,5.2) {$2$};
\node at (-1.2,5.2) {$5$};
\node at (.25,3.75) {$Q$};
\filldraw (0,4) circle[radius=2pt];
\end{tikzpicture}
\end{center}

\begin{fig} \label{here}
{\bf Cut locus of $*$}
\end{fig}
\end{minipage}
\bigskip

For $P=*$ and $Q$ the indicated degree-5 vertex, we place $(P,Q)$ in $E_3$. The vertical reflection $*'$ of $*$ has cut locus a reindexed vertical reflection of Figure \ref{here}, and we place $(*',Q')$, where $Q'$ is its degree-5 vertex, in $E_3$. Each quadrant has two analogous points in $E_3$. There are 24 quadrants, so 48 such points altogether.

Two more sets, $E_4$ and $E_5$, are required for $(P,Q)$ with $Q$ a vertex of degree 3 or 4 of the cut locus of $P$. In Figure \ref{big3}, we depict this for $Q$ in the top half of cut loci of points $P$ in the left quadrant of the 5678 face. Because the degree-5 vertex of $*$ has been placed in $E_3$, we need not worry about continuity as $*$ is approached. We place in $E_4$ all $(P,Q)$ in which $Q$ can be  approached from the 2-5 region, and depict them by solid disks. In $E_5$ we place those $(P,Q)$ not in $E_4$ which can be  approached from the 2-6 region, and depict them by open circles. The cases, in $D$, $F$, and $DF$, where $Q$ cannot be approached from the 2-5 or 2-6 regions are placed in $E_4$ or $E_5$ as indicated. 
Note that the degree-2 vertex when $P$ is on the edge $\cE$ is in $E_2$, which was already considered. The GMPRs  $\phi_4$ and $\phi_5$ choose the minimal geodesic from $P$ to $Q$ which approach $Q$ from region 2-5, 2-6, or 1-5.

\bigskip
\begin{minipage}{6in}

\begin{center}

    \begin{tikzpicture}[scale=.6]
\draw (1,11) -- (1,13) -- (2,14);
\draw (7,11) -- (7,14);
\draw (13,11) -- (13,11.7) -- (15,14);
\draw (14,-3) -- (14,-2.3) -- (12,0);
\draw (20,11) -- (20,14);
\node at (4,12.5) {$\to$};
\node at (4,5.5) {$\leftarrow$};
\node at (10.3,12.5) {$\leftarrow$};
\node at (17,12.5) {$\to$};
\node at (4,-1.5) {$\to$};
\node at (10.3,-1.5) {$\leftarrow$};
\node at (17,-1.5) {$\to$};
\node at (1,10.3) {$D$};
\node at (7,10.3) {$DF$};
\node at (13,10.3) {$F$};
\node at (20,10.3) {$FA$};
\draw (1,11.8) -- (0,11.8);
\draw (1,12.5) -- (2,12.5);
\draw (1,13) -- (0,14);
\node at (-.15,11.8) {$1$};
\node at (2.15,12.5) {$6$};
\node at (2,14.15) {$2$};
\node at (0,14.15) {$5$};
\draw (7,12) -- (6,12);
\draw (6,13.2) -- (7,13) -- (8,13.2);
\node at (5.8,12) {$1$};
\node at (5.8,13.2) {$5$};
\node at (8.2,13.2) {$6$};
\node at (7,14.2) {$2$};
\draw (13,11.7) -- (12,11.7);
\draw (13.7,12.5) -- (12.7,13);
\draw (14.4,13.3) -- (13.8,14);
\node at (11.8,11.7) {$1$};
\node at (12.5,13) {$5$};
\node at (13.7,14.1) {$2$};
\node at (15,14.2) {$6$};
\draw (19,12) -- (20,12) -- (19.3,12.8);
\draw (20,13) -- (21,13);
\node at (18.8,12) {$1$};
\node at (19.1,12.8) {$5$};
\node at (21.2,13) {$6$};
\node at (20,14.2) {$2$};
\draw (1,4) -- (1,6) -- (0,7);
\draw (0,5) -- (2,5);
\node at (1,3.3) {$\cE$};
\node at (-.2,5) {$1$};
\node at (2.2,5) {$6$};
\node at (0,7.2) {$5$};
\node at (1,8.8) {$\downarrow$};
 \filldraw (1,6) circle[radius=2pt];
 \node at (3.5,8.8) {$\searrow$};
 \node at (1.3,6) {$2$};
 \node at (7,3.3) {$DD'$};
 \node at (20,3.3) {$A$};
 \draw (7,4) -- (7,6) -- (8,7);
 \draw (7,6) -- (6,7);
 \draw (6,5) -- (8,5);
 \node at (5.75,5) {$1$};
 \node at (8.25,5) {$6$};
 \node at (6,7.25) {$5$};
 \node at (8,7.25) {$2$};
 \draw (20,4) -- (20,5) -- (21,6) -- (21,7);
 \draw (21,6) -- (22,6);
 \draw (20,5) -- (19,6) -- (19,7);
 \draw (19,6) -- (18,6);
 \node at (20,8.8) {$\uparrow$};
 \node at (22.25,6) {$6$};
 \node at (17.75,6) {$1$};
 \node at (21,7.25) {$2$};
 \node at (19,7.25) {$5$};
 \node at (1,-3.7) {$D'$};
 \node at (7,-3.7) {$D'F'$};
 \node at (14,-3.7) {$F'$};
 \node at (20,-3.7) {$F'A$};
 \node at (1,1.8) {$\uparrow$};
 \node at (20,1.8) {$\downarrow$};
 \node at (4,1.8) {$\nearrow$};
 \draw (1,-3) -- (1,-1) -- (2,0);
 \draw (1,-1) -- (0,0);
 \draw (1,-2.3) -- (2,-2.3);
 \draw (1,-1.6) -- (0,-1.6);
 \node at (2.25,-2.3) {$6$};
 \node at (-.25,-1.6) {$1$};
 \node at (0,.25) {$5$};
 \node at (2,.25) {$2$};
 \draw (7,-3) -- (7,0);
 \draw (7,-2) -- (8,-2);
 \draw (6,-.8) -- (7,-1) -- (8,-.8);
 \node at (8.25,-2) {$6$};
 \node at (5.75,-.8) {$1$};
 \node at (8.25,-.8) {$2$};
 \node at (7,.25) {$5$};
 \draw (20,-3) -- (20,0);
 \draw (21,-2) -- (20,-2) -- (21,-1);
 \draw (20,-1) -- (19,-1);
 \node at (21.25,-2) {$6$};
 \node at (21.25,-1) {$2$};
 \node at (18.75,-1) {$1$};
 \node at (20,.25) {$5$};
 \filldraw ( 7.14,-.86) circle[radius=3pt];
 \draw ( 7.14,-1.86) circle[radius=3pt];
 \filldraw ( 20.14,-1) circle[radius=3pt];
 \filldraw ( 20.14,-1.75) circle[radius=3pt];
 \filldraw ( 1,-.85) circle[radius=3pt];
 \draw (1.14,-2.2 ) circle[radius=3pt];
\draw (1.15,-1.6) circle[radius=3pt];
 \filldraw (20,5.25 ) circle[radius=3pt];
 \filldraw ( 20.84,6) circle[radius=3pt];
 \filldraw ( 19.16,6) circle[radius=3pt];
 \filldraw ( 7,6.25) circle[radius=3pt];
 \draw ( 7.16,5.16) circle[radius=3pt];
 \draw ( 1.16,5.16) circle[radius=3pt];
 \filldraw ( 19.84,12.4) circle[radius=3pt];
 \filldraw (13.7,12.7 ) circle[radius=3pt];
 \filldraw ( 14.2,13.3) circle[radius=3pt];
  \filldraw (13.3,-1.3 ) circle[radius=3pt];
 \filldraw ( 12.8,-.7) circle[radius=3pt];
 \filldraw ( 6.84,13.25) circle[radius=3pt];
 \filldraw (1,13.25 ) circle[radius=3pt];
 \draw ( 1.16,12.66) circle[radius=3pt];
  \filldraw ( 19.85,13) circle[radius=3pt];
 \filldraw (.85,11.95) circle[radius=3pt];
 \draw (6.85,12.15) circle[radius=3pt];
 \draw (14.15,-2.15) circle[radius=3pt];
\draw (14,-2.3) -- (15,-2.3);
\draw (13.3,-1.5) -- (14.3,-1);
\draw (12.6,-.7) -- (13.2,0);
\node at (15.2,-2.3) {$6$};
\node at (14.5,-1) {$2$};
\node at (13.3,.1) {$5$};
\node at (12,.2) {$1$};
\draw (12.87,11.85) circle[radius=3pt];
\end{tikzpicture}
\end{center}

\begin{fig} \label{big3}
{\bf Approach to vertices of cut loci}
\end{fig}
\end{minipage}
\bigskip

Each arrow in Figure \ref{big3} represents points $P$ in a region approaching points in its boundary. A segment in a cut locus shrinks to a point. Continuity of each separate function $\phi_i$ should be clear. 

All quadrants of all faces are handled similarly, using permutations of corner-point numbers. In particular, if $P$ is in the analogue of the large region $A$ in any quadrant, and $Q$ is a vertex of the cut locus of $P$, then $(P,Q)$ is in $E_4$. Since regions $A$ are the only regions abutting a diagonal, (see Figures \ref{figA} or \ref{DF}) if, for the degree-3 and degree-4 vertices $Q$ of the cut locus of points $P$ in the diagonals of the quadrants, we place $(P,Q)$ in $E_5$, then there is no worry about continuity of $\phi$ functions at these points, as long as we make consistent choices. The cut locus of the center  of a quadrant has four arms emanating from a central vertex, with a degree-2 vertex on each arm. In the 5678 face, it is obtained from the cut locus of the diagonal pictured in Figure \ref{AA} by collapsing the arms from 1 and 3 to a point. We make an arbitrary choice of $\phi_5(P,Q)$ when $Q$ is the degree-4 vertex of the center $P$ of a face, and then choose $\phi_5(P,Q)$ compatibly when $Q$ is the degree-4 vertex of points $P$ on the diagonals of the face.

In the paragraph following Figure \ref{big1}, we described how bottom halves of cut loci are determined from top halves of cut loci, and we put these $(P,Q)$ with $Q$ a vertex of degree 3 or 4 in the bottom half of the cut locus of $P$ in sets $E_i$ with GMPRs $\phi_i$, $4\le i\le5$, analogously to what was done for the top halves. 

The cube is composed of twelve regions such as that in Figure \ref{both}, each bounded by half-diagonals of faces, and symmetrical about an edge of the cube. For cut-locus vertices of degree 3 or 4, the GMPRs on the diagonals are in separate sets from those on the $A$-regions abutting them, and so the twelve regions can be considered separately. Once we have defined the GMPRs for the region containing the 5-8 edge, GMPRs on the other regions can be defined similarly, using permutations of corner-point numbers.

\bigskip
\begin{minipage}{6in}

\begin{center}

    \begin{tikzpicture}[scale=.4]
\draw (0,0) -- (5,5) -- (0,10) -- (0,0) -- (-5,5) -- (0,10);
\draw (0,0) to[out=110,in=-110] (0,10);
\draw (0,0) to[out=70,in=-70] (0,10);
\node at (0,-.4) {$5$};
\node at (0,10.4) {$8$};
\node at (3,5) {$A$};
\node at (-3,5) {$\widetilde A$};

\end{tikzpicture}
\end{center}

\begin{fig} \label{both}
{\bf A subset of the cube}
\end{fig}
\end{minipage}
\bigskip

The cut locus of a point $\widetilde{P}$ on the left half of Figure \ref{both} is obtained from that of its horizontal reflection by applying the permutation $(1\ 6)(4\ 7)$ and flipping horizontally. In Figure \ref{flip}, we show top halves of cut loci for points in the reflection of the edge $\cE$ and of the regions abutting it, together with their GMPRs for vertices of degree $\ge3$. Note that $\cE=\widetilde{\cE}$, so they have the same cut loci, but the depictions of them from the star unfolding are different depending on whether they are the left or right edge.

\bigskip
\begin{minipage}{6in}

\begin{center}

    \begin{tikzpicture}[scale=.6]
\draw (0,0) -- (0,3) -- (1,4);
\draw (-1,1) -- (0,1) -- (0,2) -- (1,2);
\draw (0,3) -- (-1,4);
\node at (-1.3,1) {$1$};
\node at (1.3,2) {$6$};
\node at (-1.3,4) {$2$};
\node at (1.3,4) {$5$};
\node at (0,-.6) {$\widetilde{D'}$};
\filldraw (-.2,1.2) circle[radius=3pt];
\draw (.2,15) circle[radius=3pt];
\draw (0,7) -- (0,9) -- (1,10);
\draw (-1,8) -- (1,8);
\draw (0,9) -- (-1,10);
\node at (-1.3,8) {$1$};
\node at (1.3,8) {$6$};
\node at (-1.3,10) {$2$};
\node at (1.3,10) {$5$};
\node at (0,6.4) {$\widetilde{DD'}$};
\draw (0,13) -- (0,16) -- (1,17);
\draw (0,16) -- (-1,17);
\draw (-1,15) -- (0,15) -- (0,14) -- (1,14);
\node at (1.3,14) {$6$};
\node at (-1.3,15) {$1$};
\node at (-1.3,17) {$2$};
\node at (1.3,17) {$5$};
\node at (0,12.4) {$\widetilde{D}$};
\node at (0,10.8) {$\downarrow$};
\node at (0,4.8) {$\uparrow$};
\node at (2.5,8) {$\to$};
\node at (2.5,10.8) {$\searrow$};
\node at (2.5,4.8) {$\nearrow$};
\draw (5,7) -- (5,9) -- (6,10);
\draw (4,8) -- (6,8);
\node at (3.7,8) {$1$};
\node at (6.3,8) {$6$};
\node at (6.3,10) {$5$};
\node at (4.7,9) {$2$};
\node at (5,6.2) {$\widetilde{\cE}$};
\draw ( .14,2.14) circle[radius=3pt];
\draw ( .14,8.14) circle[radius=3pt];
\draw ( .14,14.14) circle[radius=3pt];
\draw ( 5.14,8.14) circle[radius=3pt];
\filldraw (0,3.24 ) circle[radius=3pt];
\filldraw ( 0,9.24) circle[radius=3pt];
\filldraw ( 0,16.24) circle[radius=3pt];
\end{tikzpicture}

\end{center}

\begin{fig} \label{flip}
{\bf Horizontal reflection}
\end{fig}
\end{minipage}
\bigskip

The sets $E_i$ and functions $\phi_i$, $4\le i\le5$, for the left side of Figure \ref{both} are defined like those on the primed (or unprimed) version on the right side, with 2 and 5 interchanged. Compare $\widetilde{D'}$ (resp.~$\widetilde{D}$) in Figure \ref{flip} with $D$ (resp.~$D'$) in Figure \ref{big3}. This completes the proof of Theorem \ref{thm1}.

\section{Lower bound}\label{lower}
In this section we prove the following result, which is the lower bound in Theorem \ref{mainthm}. The method is similar to that developed by Recio-Mitter in \cite{david} and applied by the author in \cite{tet}.
\begin{thm} If $X$ is a cube, it is impossible to partition $X\times X$ into sets $E_i$, $1\le i\le4$, with a GMPR $\phi_i$ on $E_i$.\end{thm}
\begin{proof}
Assume such a decomposition exists. Note that the specific $E_i$ of the previous section are not relevant here.
Let $V_i$ be the corner point numbered $i$ in our treatment of the cube. The cut locus of $V_8$ is as in the left side of Figure \ref{four}.  It consists of edges from $V_2$ to corner points 1, 3, and 6, and diagonals from $V_2$ to corner points 4, 5, and 7.

Let $E_1$ be the set containing $(V_8,V_2)$, and suppose $\phi_1(V_8,V_2)$ is the geodesic passing between $V_3$ and $V_4$. Other cases can be handled in the same way, using a permutation of corner points.

 Points $P$ on the curve $DE$ of Figure \ref{figB} have top half of cut loci as in Figure \ref{24}. (This is part of the curve $DF$ in Figure \ref{DF}.)

\bigskip
\begin{minipage}{6in}

\begin{center}

    \begin{tikzpicture}[scale=.7]
    \draw (0,0) -- (0,3);
    \draw (0,1) -- (-2,3);
    \draw (-2,4) -- (0,2) -- (2,3);
    \node at (0,3.3) {$2$};
    \node at (2.3,3) {$6$};
    \node at (-2.3,4) {$5$};
    \node at (-2.3,3) {$1$};
    \node at (1,1.7) {$\a$};
    \node at (1,3.1) {$\b$};
    \node at (-.8,3.5) {$\g$};
    \node at (-1,2.4) {$\d$};
    \node at (0.2,1.8) {$Q$};
     \filldraw  (0,2 ) circle[radius=2pt];
\end{tikzpicture}
\end{center}

\begin{fig} \label{24}
{\bf Top half of cut locus of points on curve $DE$}
\end{fig}
\end{minipage}
\bigskip

Let $Q$ be the vertex of degree 4, and $\a$, $\b$,$\g$, and $\d$ the four regions of approach to $Q$, as indicated in the figure, which varies with $P$. As $P$ approaches $V_8$ along $DE$, Figure \ref{24} approaches the top half of the cut locus of $V_8$ (Figure \ref{four}); the segment from $Q$ to 2 shrinks to the point $V_2$, and the other vertical segment collapses, too. Suppose there were a sequence of points $P_n$ on $DE$ approaching $V_8$ with $Q_n$ the point $Q$ in Figure \ref{24} and $(P_n,Q_n)\in E_1$. Then $\phi_1(P_n,Q_n)$ would approach
$\phi_1(V_8,V_2)$, but this is impossible, since they pass through different regions. Therefore there must be a sequence $P_n$ on $DE$ approaching $V_8$ for which $(P_n,Q_n)$ is in a different set, $E_2$, and restricting further, we may assume that $\phi_2(P_n,Q_n)$ all pass through the same region, $\a$, $\b$, $\g$, or $\d$.

Points in region $D$ have top half of cut locus as in Figure \ref{26}. See Figure \ref{big1}.

\bigskip
\begin{minipage}{6in}

\begin{center}

    \begin{tikzpicture}[scale=.7]
    \draw (0,0) -- (0,3) -- (1,4);
    \draw (0,3) -- (-2,5);
    \draw (0,2) -- (2,4);
    \draw (0,1) -- (-2,3);
    \node at (2.3,4) {$6$};
    \node at (1,4.3) {$2$};
    \node at (-2.3,3) {$1$};
    \node at (-2.3,5) {$5$};
    \node at (.35,1.8) {$Q_\g$};
    \node at (0,3.6) {$Q_\a$};
     \filldraw  (0,2 ) circle[radius=2pt];
      \filldraw  ( 0,3) circle[radius=2pt];
\end{tikzpicture}
\end{center}

\begin{fig} \label{26}
{\bf Top half of cut locus of points in region $D$}
\end{fig}
\end{minipage} 
\bigskip

Let $Q_\a$ and $Q_\g$ be the indicated vertices in Figure \ref{26}. If $\phi_2(P_n,Q_n)$ passes through region $\a$ (resp.~$\g$) in Figure \ref{24}, consider a sequence of points $P_{n,m}$ in region $D$ approaching $P_n$, and let the associated cut-locus points $Q_{n,m}$ be $Q_a$ (resp.~$Q_\g$). Such a sequence $(P_{n,m},Q_{n,m})$ cannot have a convergent subsequence in $E_2$, since, if it did, reindexing, $\phi_2(P_{n,m},Q_{n,m})\to \phi_2(P_n,Q_n)$, but paths going to $Q_\a$ (resp.~$Q_\g$) cannot approach a path passing through region $\a$ (resp.~$\g$) in Figure \ref{24}. So we may restrict to points $(P_{n,m},Q_{n,m})$ not in $E_2$, and restricting further, we may assume they are all in the same $E_i$. If $i=1$, then $(P_{n,n},Q_{n,n})$\footnote{This should really be $(P_{n,n'},Q_{n,n'})$ for some $n'\ge n$, but we will simplify the notation as we have here and subsequently.} would approach $(V_8,V_2)$ and would have $\phi_1(P_{n,n},Q_{n,n})\to \phi_1(V_8,V_2)$, which is impossible since these paths pass through different regions.\footnote{Here, as in many other parts of this proof, when we say ``pass through'' a region, we mean, of course, that the portion of the curve as it approaches the limit point  passes through the region.} Thus all $(P_{n,m},Q_{n,m})$ must be in either $E_3$ or $E_4$, and we may assume they are all in $E_3$.

A similar argument works if all $\phi_2(P_n,Q_n)$ pass through region $\b$ or $\d$ in Figure \ref{24}, using points $P_{n,m}$ in region $E$ of Figure \ref{figB} approaching $P_n$, and $Q_{n,m}$ the points $Q_\b$ or $Q_\d$ in Figure \ref{E}, which depicts the top half of the cut locus of points in region $E$ of Figure \ref{figB}. (Region $E$ of Figure \ref{figB} is part of region $F$ of Figure \ref{DF}.)
Thus we conclude that all $(P_{n,m},Q_{n,m})$ are in $E_3$, regardless of whether $\phi_2(P_n,Q_n)$ passed through $\a$, $\b$, $\g$, or $\d$.

\bigskip
\begin{minipage}{6in}

\begin{center}

    \begin{tikzpicture}[scale=.7]
 \draw (0,0) -- (0,4);
 \draw (0,1) -- (-2,4);
 \draw (0,2) -- (-2,5);
 \draw (0,3) -- (2,4);
 \node at (.4,2.8) {$Q_\d$};
    \node at (.4,1.9) {$Q_\b$};
     \filldraw  (0,2 ) circle[radius=2pt];
      \filldraw  ( 0,3) circle[radius=2pt];
\node at (2.3,4) {$6$};
\node at (0,4.3) {$2$};
\node at (-2.3,4) {$1$};
\node at (-2.3,5) {$5$};

\end{tikzpicture}
\end{center}

\begin{fig} \label{E}
{\bf Top half of cut locus of points in region $E$}
\end{fig}
\end{minipage} 
\bigskip

Suppose $\phi_2(P_n,Q_n)$ pass through region $\a$ in Figure \ref{24}, and $Q_{n,m}$ were the points $Q_\a$ in Figure \ref{26}. An argument similar to the one that we will provide works if $\a$ is replaced by $\b$, $\g$, or $\d$. All that matters is that the vertex $Q_\a$ (or its analogue) has degree 3. In Figure \ref{D}, we isolate the relevant portion of Figure \ref{26}, with $Q_{n,m}$ at the indicated vertex.

\bigskip
\begin{minipage}{6in}

\begin{center}

    \begin{tikzpicture}[scale=.4]
     \filldraw  (0,2 ) circle[radius=3.5pt];
     \draw (0,0) -- (0,2) -- (-2,4);
     \draw (0,2) -- (2,3);
     \node at (2.3,3) {$2$};
     \node at (-2.3,4) {$5$};

\end{tikzpicture}
\end{center}

\begin{fig} \label{D}
{\bf A portion of Figure \ref{26}}
\end{fig}
\end{minipage} 
\bigskip

We may assume, after restricting, that all $\phi_3(P_{n,m},Q_{n,m})$ pass through the same one of the three regions in Figure \ref{D}, which we call region $R$. For a sequence $Q_{n,m,\ell}$ approaching $Q_{n,m}$ on the edge not bounding $R$, $(P_{n,m},Q_{n,m,\ell})$ cannot have a convergent subsequence in $E_3$, since $\phi(P_{n,m},Q_{n,m,\ell})$ cannot pass through $R$. Restricting more, we may assume that all $(P_{n,m},Q_{n,m,\ell})$ are in the same $E_i$, with $i\ne3$. If $i=2$, then $\phi_2(P_{n,m},Q_{n,m,m})$ approaches $ \phi_2(P_n,Q_n)$, but geodesics from $P_{n,m}$ to points close to $Q_\a$ in Figure \ref{26} ultimately are above the arm from $Q$ to corner point 6 in Figure \ref{24}, while $\phi_2(P_n,Q_n)$ is below it. (Recall that the cut locus in Figure \ref{24} is approached by those in Figure \ref{26}.)
 So $i\ne2$. Also, $i$ cannot equal 1, because if so, $\phi_1(P_{n,n},Q_{n,n,n})\to \phi_1(V_8,V_2)$, but the latter is between vertices 3 and 4 in the lower half of the cut locus. Therefore $i=4$.

We may assume, after restricting, that all the $\phi_4(P_{n,m},Q_{n,m,\ell})$ come from the same side of the edge in Figure \ref{D} which contains the points $Q_{n,m,\ell}$. Choose points $Q_{n,m,\ell,k}$ in the complement of the cut locus of $P_{n,m}$ on the opposite side of the edge, and converging to $Q_{n,m,\ell}$. Restricting, we may assume that all $(P_{n,m},Q_{n,m,\ell,k})$ are in the same $E_i$. Note that $\phi_i(P_{n,m},Q_{n,m,\ell,k})$ is the unique geodesic between these points. This $i$ cannot equal 4 since $\phi_i(P_{n,m},Q_{n,m,\ell,k})$ and
   $\phi_4(P_{n,m},Q_{n,m,\ell})$ approach the edge from opposite sides. It cannot equal 3 since $\phi_i(P_{n,m},Q_{n,m,\ell,\ell})$ and
   $\phi_3(P_{n,m},Q_{n,m})$ approach the vertex in Figure \ref{D} from different regions. It cannot equal 2 since $\phi_i(P_{n,m},Q_{n,m,m,m})$ and
   $\phi_2(P_{n},Q_{n})$ approach the vertex in Figure \ref{24} from different regions. And, it cannot equal 1 since $\phi_i(P_{n,n},Q_{n,n,n,n})$ and
   $\phi_1(V_8,V_2)$ approach  vertex 2 from different regions. Therefore a fifth $E_i$ is required.
\end{proof}
\def\line{\rule{.6in}{.6pt}}

\end{document}